\documentclass[12pt]{amsart}
\usepackage{amssymb}
\newcommand{\h}{\mathfrak{h}}
\newcommand{\Z}{\mathbb{Z}}
\newcommand{\C}{\mathbb{C}}
\newcommand{\M}{\mathcal{M}}
\newcommand{\Dcal}{\mathcal{D}}
\newcommand{\Loc}{\operatorname{Loc}}
\newcommand{\OCat}{\mathcal{O}}
\newcommand{\Hom}{\operatorname{Hom}}
\newcommand{\Ext}{\operatorname{Ext}}
\newtheorem{Thm}{Theorem}[section]
\newtheorem{Prop}[Thm]{Proposition}
\newtheorem{Cor}[Thm]{Corollary}
\newtheorem{Lem}[Thm]{Lemma}
\theoremstyle{definition}

\newtheorem{Rem}[Thm]{Remark}
\newtheorem{Conj}[Thm]{Conjecture}
\numberwithin{equation}{section}
\address{Department
of Mathematics, Northeastern University, Boston MA 02115 USA}
\email{i.loseu@neu.edu}
\title{Abelian localization for cyclotomic Cherednik algebras}
\author{Ivan Losev}
\begin{document}
\begin{abstract}
In this paper we prove the abelian localization theorem for modules over cyclotomic Rational Cherednik algebras.
\end{abstract}
\maketitle
\section{Introduction}
The Beilinson-Bernstein theorem proved in \cite{BB} is one of the most fundamental results in the representation theory of semisimple
Lie algebras. The theorem relates two categories. The first one is the  module category of $\mathcal{U}_\lambda$, the central reduction
of the universal enveloping algebra $U(\mathfrak{g})$ corresponding to $\lambda\in \mathfrak{h}^*$, where $\mathfrak{g}$
is a semisimple Lie algebra over $\C$ and $\mathfrak{h}\subset \mathfrak{g}$ is a Cartan subalgebra. The second category is the category
of quasi-coherent (sheaves of) modules over the sheaf $\mathcal{D}_{G/B}^\lambda$ of twisted differential operators on $G/B$, the flag variety of $\mathfrak{g}$. There are two natural functors between the two categories, the global section functor $\Gamma:\mathcal{D}_{G/B}^\lambda\operatorname{-mod}\rightarrow \mathcal{U}_{\lambda}\operatorname{-mod}$ and the localization
functor $\Loc: \mathcal{U}_{\lambda}\operatorname{-mod}\rightarrow \mathcal{D}_{G/B}^\lambda\operatorname{-mod}$.
The localization theorem states that the functors are mutually (quasi-)inverse equivalences if and only if $\langle\lambda+\rho,\alpha^\vee\rangle\not \in \Z_{\leqslant 0}$ for any positive coroot $\alpha^\vee$. Here, as usual,
$\rho$ is half the sum of all positive roots.

In the last 10 years there was a renewed interest to localization results in a wider representation theoretic context.
Namely, we can view $\mathcal{D}^\lambda_{G/B}$ as a quantization of the cotangent bundle $T^*(G/B)$ and $\mathcal{U}_\lambda$
as a quantization of the nilpotent cone $\mathcal{N}\subset \mathfrak{g}^*$. We have the Springer resolution morphism
$T^*(G/B)\rightarrow \mathcal{N}$. Now given any conical singular Poisson variety $X_0$ and its conical symplectic resolution
of singularities $X$, we can define quantizations $\mathcal{A}_\lambda$ of $X$ for $\lambda\in H^2_{DR}(X)$. They are filtered sheaves of algebras
in the conical topology on $X$, see \cite{BPW} for a related but different formalism. Then the global sections algebra $A_\lambda:=\Gamma(\mathcal{A}_\lambda)$ is a quantization of $X_0$ and one can ask for which $\lambda$ the functor
$\Gamma:\mathcal{A}_\lambda\rightarrow A_\lambda$ is an equivalence of categories.

The answer is known in some cases. Historically, the first case beyond the cotangent bundles considered
was when $X$ is the Hilbert scheme $\operatorname{Hilb}_n(\C^2)$, see \cite{GS}. The algebra $A_\lambda$
in this case is the spherical subalgebra in the type A Rational Cherednik algebra $H_\lambda$.
In {\it loc.cit.},  quantizations of $X$ were described using a different setting -- that of $\Z$-algebras.
Later Kashiwara and Rouquier, \cite{kashrouq}, proved the localization theorem using the setting of $\mathcal{W}$-algebras
that is more or less equivalent to what was described in the previous paragraph. Recently, it was shown
in \cite{BPW} that the two settings are equivalent.

Subsequently, more cases were worked out. In \cite{BPW}, it was shown that the localization theorem holds for $\lambda$
provided it is ``large enough'' but their approach does not allow to find a precise description.
The case when $X$ is a Slodowy variety (and $A_\lambda$ is the central reduction of a finite W-algebra) was fully solved in \cite{DK} (and earlier in \cite{Ginzburg_HC} in the $\Z$-algebra setting, a special case of
minimal resolutions of Kleinian singularities was settled yet earlier in \cite{Boyarchenko}). Some partial results in the case when
$X$ is a hypertoric variety were obtained in \cite{BK}. The case when $X$ is  a more general Hamiltonian reduction
(the most interesting case here is that of Nakajima quiver varieties; $\operatorname{Hilb}_n(\C^2)$, the Slodowy
varieties of type A, as well as the minimal resolutions of Kleinian singularities are special cases of quiver varieties)
was treated in \cite{MN,MN_ab}. In \cite{MN}
it was shown that $R\Gamma$ is a derived equivalence if and only if the algebra $A_\lambda$ has finite homological dimension
(such parameters are called {\it regular}, non-regular parameters are called {\it singular}).
Further, in \cite{MN_ab}, a sufficient condition on $\lambda$ for $\Gamma$
to be exact was found. This condition does not seem to be optimal. Together with Bezrukavnikov,
the author stated \cite[Conjecture 9.2]{BL} that describes the singular locus and the locus
whether the abelian localization fails. Both loci are unions of hyperplanes. According to that
conjecture, the abelian localization can hold even outside the locus determined by McGerty and Nevins.

The goal of this paper is to generalize the results of \cite{GS,kashrouq} to more general Cherednik algebras.
Namely, we will prove an abelian
localization theorem for  cyclotomic Rational Cherednik algebras. Those correspond to
groups $W=G(\ell,1,n):=\mathfrak{S}_n\ltimes (\Z/\ell \Z)^n$ (type $A$ Cherednik algebras correspond to
$\ell=1$). Also the spherical subalgebras arise as $A_\lambda$ for symplectic resolutions of $\C^{2n}/W$.
Our result should imply \cite[Conjecture 9.2]{BL} but we are not going to elaborate on that in
the present paper as the reduction involves some combinatorics.

There are basically three reasons which make this case more tractable than the general one. First of all,
one can deal with the whole Cherednik algebra  (that quantizes the skew-group ring $\C[\C^{2n}]\#W$ instead of
$\C[\C^{2n}]^W$) to be denoted by $H_p$ (above we have used $\lambda$ for the parameter, but we will use $p$
from now on). There is a sheaf analog  $\mathcal{H}_p$ of $H_p$, the sheaf of endomorphisms of the right $\mathcal{A}_p$-module
quantizing a {\it Procesi bundle} $\mathcal{P}$. One advantage of dealing with full Cherednik algebras is that the functor $R\Gamma$
is always an equivalence, \cite{GL}. The second reason is that both for $\mathcal{A}_p$ and for $H_p$
there are so called categories $\mathcal{O}$ that have highest weight structure. It is this structure and an observation
that one can describe Verma modules using Procesi bundles (made in \cite{VV_proof}) that makes the case of cyclotomic
Rational Cherednik algebras tractable. Finally, the singular locus here is known, essentially due to \cite{DG}.

In fact, categories $\mathcal{O}$ make sense whenever there is a Hamiltonian $\C^\times$-action with finitely many fixed points
on $X$, see \cite{BLPW}. It seems that in this case proving localization theorems should be easier than in general. We make
some speculations on how to extend techniques used in the present paper to a more general case. In fact, in a subsequent
paper we plan to prove the derived and abelian localization theorems in another important special case: when $X$
is the Gieseker moduli space.

The paper is organized as follows. In Section \ref{S:prelim}, we briefly recall the construction of symplectic resolutions
of $\C^{2n}/W$, of the Cherednik algebra $H_p$, its spherical subalgebra $A_p=eH_pe$,
as well as of their local versions $\mathcal{H}_p,\mathcal{A}_p$. Next, we recall
definitions of categories O, $O_p$ for $H_p$ and $\mathcal{O}_p$ for $\mathcal{H}_p$ and provide some details
on the highest weight structures on these categories. Then we recall a derived localization theorem from
\cite{GL} and describe the singular parameters. Finally, in \ref{SS:main_res} we state the main result
of this paper, Theorem \ref{Thm:equiv}, that gives a sufficient condition for $\Gamma: \mathcal{H}_p\operatorname{-mod}\rightarrow
H_p\operatorname{-mod}$ to be a category equivalence. Then we state a corollary that says, in particular, that
for any $p$ there is a resolution such that the abelian localization holds in our setting.

These claims are proved in Section \ref{S:proof}. First, following an idea that was already used in  \cite{VV_proof}, we relate the local
analogs of Verma modules in the category $\mathcal{O}_p$ to the standard objects in that category. Based on this,
we prove that the localization holds for categories $\mathcal{O}$. Then we deduce the localization theorem for the
full categories of modules. Finally, we prove a corollary of Theorem \ref{Thm:equiv}.

In Section \ref{S:open} we speculate on how to generalize our results to more general symplectic resolutions
that admit a Hamiltonian torus action with finitely many fixed points.

{\bf Acknowledgements}: I am supported
by the NSF under Grant  DMS-1161584.

\section{Preliminaries}\label{S:prelim}
\subsection{Resolutions}\label{SS:res}
Here we are going to recall a construction of symplectic resolutions of $\C^{2n}/W$ via Hamiltonian reduction.
Recall that $G(\ell,1,n):=\mathfrak{S}_n\ltimes (\Z/\ell\Z)$ acts naturally on $\C^n$ and hence also on
$\C^n\oplus (\C^n)^*=\C^{2n}$.

Namely, let $Q$ be a cyclic quiver with $\ell$ vertices, in other words, the Dynkin diagram of  affine type $\tilde{A}_{\ell-1}$
with some orientation. We number the vertices of $Q$ by numbers from $0$ to $\ell-1$ cyclically.
Let $\delta=(1,\ldots,1)\in \C^{Q_0}$ be the indecomposable imaginary root. We consider the space
$R=\operatorname{Rep}(n\delta,\epsilon_0)$ of representations of $Q$ of dimension $n\delta$ with one-dimensional framing at
vertex $0$. The group $G:=\prod_{i\in Q_0}\operatorname{GL}(n\delta_i)$ naturally acts on this space. Fix a stability
condition $\theta\in \Z^{Q_0}$ and consider the Nakajima quiver variety $\M^\theta:=\mu^{-1}(0)^{\theta-ss}/\!/G$, where
$\mu: T^*R\rightarrow \mathfrak{g}^*$ is a natural moment map.

The variety $\M^0$ is naturally identified with $\h\oplus\h^*/W$, see, for example, \cite{GG}.
A variety $\M^\theta$ with $\theta$ generic is smooth and symplectic.
The generic locus can be described explicitly, see \cite[Theorem 2.8]{Nakajima}. Namely, $\theta$ is generic if and only if $\sum_{i=0}^{\ell-1}\theta_i\neq 0$ and $\theta_i-\theta_j\neq m\sum_{k=0}^{\ell-1}\theta_k$ for all different
$i,j\in \{1,\ldots,\ell-1\}$ and all integers  $m$ with $|m|<n$. 

The variety $\M^\theta$ comes with a natural projective morphism $\M^\theta\rightarrow \M^0=\h\oplus \h^*/W$ that is a resolution of singularities for generic $\theta$. The latter follows, for example,  from considerations in \cite{GG}.

Let $T_1,T_2$ denote one-dimensional tori $\C^\times$.
We have a $T_1\times T_2$-action on $\M^\theta,\M^0$ that comes from the following action on $T^*R=R\oplus R^*$:
$(t_1,t_2)(r,\alpha)=(t_1^{-1}t_2^{-1} r, t_1^{-1}t_2\alpha), r\in R,\alpha\in R^*$. The morphism $\M^\theta\rightarrow \M^0$
is clearly equivariant. Under the identification of $\M^0$ with $(\h\oplus \h^*)/W$, the action of $T_1\times T_2$
coincides with an action induced from $(t_1,t_2)(y,x)=(t_1^{-1}t_2^{-1}y, t_1^{-1}t_2 x)$.

We note that the fixed points of the $T_2$-action on $\M^\theta$ are all isolated. The set $(\M^\theta)^{T_2}$
is identified with the set of irreducible $W$-modules as explained in \cite{Gordon}. The latter is in a natural
bijection with the set $\operatorname{P}_\ell(n)$ of the $\ell$-multipartitions of $n$, where we use the numbering conventions of \cite{GL}.

\subsection{Algebras and sheaves}\label{SS:alg}
Our convention on parameters for a cyclotomic Rational Cherednik algebra is as in \cite{GL}. Namely, we have parameters $p=(\kappa,h_0,\ldots,h_{\ell-1})$, where the numbers $h_0,\ldots,h_{\ell-1}$ are defined up to a common summand.

Then we can form the algebra $H_p$, a Rational Cherednik algebra for the the group $W$,
that is a quotient of $T(\h\oplus \h^*)\#W$ by certain relations, here $\h=\C^n$ is the reflection representation of $W$.
Inside $H_p$ we have the averaging idempotent $e=\frac{1}{|W|}\sum_{w\in W}w$. Then we can consider the spherical subalgebra
$eH_pe$.

Both algebras $H_p,eH_pe$ are filtered: their associated graded algebras are $\C[\h\oplus \h^*]\#W$ and $\C[\h\oplus \h^*]^W$, respectively.
Now fix a generic stability condition $\theta=(\theta_0,\ldots,\theta_{\ell-1})$ and form the corresponding variety
$\M^\theta$. As in \cite{quant}, we have a sheaf (in conical topology) of filtered algebras $\mathcal{A}_p$ that quantizes
the structure sheaf $\mathcal{O}_{\M^\theta}$ (meaning that the $T_1$-equivariant Poisson sheaves $\operatorname{gr}\mathcal{A}_p$ and $\mathcal{O}_{\M^\theta}$ are identified) and such that $\Gamma(\mathcal{A}_p)=eH_pe$ (and all higher cohomology of $\mathcal{A}_p$ automatically vanish). We remark that the sheaf $\mathcal{A}_p$ is equivariant with respect to $T_2$ (and the action of $T_2$ on $\mathcal{A}_p$ is Hamiltonian).

We also have a local version of $H_p$. Let $\mathcal{P}$ be a distinguished Procesi bundle on $\M^\theta$ (see \cite{VV_proof}
for the definition of that, this is the same sheaf as used in \cite{BF}).
Then we can deform $\mathcal{P}$ to a right $\Dcal_p$-module. Let $\mathcal{H}_p$ denote the endomorphism
sheaf of that module. Then $\Gamma(\mathcal{H}_p)=H_p$ and the higher cohomology of $\mathcal{H}_p$ vanish.

\subsection{Categories}\label{SS:cat}
We consider the categories $\mathcal{O}$ over the algebra $H_p$ (to be denoted by $O_p$) and over $\mathcal{H}_p$
(to be denoted by $\mathcal{O}_p$).

We remark that the sheaves of algebras  $\mathcal{H}_p$ and $\mathcal{A}_p$ are Morita
equivalent. Under this Morita equivalence, the category $\mathcal{O}_p$
corresponds to the category $\mathcal{O}$ for $\mathcal{A}_p$ defined in \cite{BLPW}.
Recall how the latter is defined.
Let $Y$ stands for the contracting locus for the $T_2$-action: the set of all points $y\in \M^\theta$
such that $\lim_{t\rightarrow 0}t.y$ exists. The category $\mathcal{O}_p$  consists of all
$\mathcal{A}_p$-modules that are supported on $Y$ that have a global good filtration (``good'' means
that the associated graded is a coherent sheaf) that is stable under the quantum Hamiltonian for
the $T_2$-action on $\mathcal{A}_p$.

The category $\mathcal{O}_p$ is highest weight. Its irreducible (and standard) objects are
parameterized by the $\C^\times$-fixed points. Let $z_\lambda$ denote the  point corresponding to an
$\ell$-multipartition $\lambda$. Let $Y_\lambda$
denote the contracting component of $z_\lambda$. We define the geometric order on $\operatorname{P}_\ell(n)$
by setting $\lambda\leqslant \mu$ if $z_\lambda\in \overline{Y}_\mu$ (we write $\lambda\leqslant^\theta \mu$ if we need
to indicate that this is the geometric order defined for the stability condition $\theta$).
The claim that $\mathcal{O}_p$ is a highest weight category with respect to this order was established in \cite{BLPW}.

By definition, the category $O_p$  consists of all $H_p$-modules that are finitely generated over $S(\h^*)$
and where the action of $\h$ is locally nilpotent. Here is an example of an object in this category:  for a multipartition $\lambda$ thought as an
irreducible $W$-module we can consider the Verma module $\Delta_p(\lambda):=\operatorname{Ind}_{S(\h)\#W}^{H_p}\lambda$.
It was checked in \cite{GGOR} that $O_p$ is a highest weight category with standard objects $\Delta_p(\lambda)$,
where the order is as follows. There is an Euler element $h\in H_p$ having the property that $[h,x]=x, [h,y]=-y, [h,w]=0$
for all $x\in \h^*, y\in \h, w\in W$. This element has the form $\sum_{i=1}x_i y_i+ C$, where $C$ is a central element
in $\C W$ (depending on $p$). Define $c_p(\lambda)$ to be the scalar by which $C$ acts on $\lambda$. Define the ordering
$\leqslant^p_c$ on $\operatorname{P}_\ell(n)$ as follows: $\lambda\leqslant^p_c \mu$ if $\lambda=\mu$ or $c_p(\lambda)-c_p(\mu)\in \Z_{>0}$.
Then $O_p$ is a highest weight category with respect to $\leqslant_c^p$.

In the study of the category $\mathcal{O}_p$, there are two cases that need to be treated differently: $\kappa=0$ and $\kappa\neq 0$.
Consider (a more interesting) case $\kappa\neq 0$.
As was shown in \cite{Griffeth}, one can take a weaker ordering than $\leqslant_c^p$.
Namely, to a box $b$ in the $i$th diagram with coordinates
$(x,y)$ we assign the number $\operatorname{cont}_p(b)=h_i+\kappa(y-x)$. We say that $b\sim b'$  if $\operatorname{cont}_p(b)-\operatorname{cont}_p(b')\in \Z+\frac{i-i'}{\ell}$. We say that $b<b'$ if $b\sim b'$
and $ \operatorname{cont}_p(b)-\operatorname{cont}_p(b')\in \mathbb{Q}_{<0}$.  For multipartitions $\lambda,\lambda'$ of $n$, we say that $\lambda\leqslant^p \lambda'$ if one can order the boxes $b_1,\ldots,b_n$ of $\lambda$ and $b_1',\ldots,b_n'$ of $\lambda'$ in such a way that $b_i\leqslant b_i'$. It is known, thanks to computations of $c_p(\lambda)$ done in \cite{rouqqsch}, that $\lambda\leqslant^p \lambda'$ implies $\lambda\leqslant^p_c \lambda'$.

We remark that for $\kappa=0$, the category $O_p$ is the $\mathfrak{S}_n$-equivariant category $\mathcal{O}$
for the algebra $H_p(1)^{\otimes n}$, where $H_p(1)$ is the cyclotomic Rational Cherednik algebra for $n=1$
depending on the parameters $h_0,\ldots,h_{\ell-1}$. The category $\mathcal{O}$ for $n=1$ is very easy and hence,
to some extent, the case of $\kappa=0$ is completely understood. However, for us, it presents some complications
of combinatorial nature, so we are not going to consider that case.


\subsection{Derived equivalence}\label{SS:dereq}
We have the functor $\Gamma: \mathcal{H}_p\operatorname{-mod}\rightarrow H_p\operatorname{-mod}$ of taking global sections.
It has a left adjoint, the localization functor, $\Loc:=\mathcal{H}_p\otimes_{H_p}\bullet$. Since an $H_p$-module lies in
$O_p$ if and only if it is finitely generated over $S(\h^*)$ and can be made $\C^\times$-equivariant, we see
that $\Gamma,\Loc$ restrict to functors between $O_p$ and $\mathcal{O}_p$.

We also can consider the derived functors $R\Gamma, L\Loc$ between the categories $D^b(\mathcal{H}_p\operatorname{-mod})$
and $D^b(H_p\operatorname{-mod})$. It follows from results of \cite[Section 5]{GL} that these functors are mutually inverse
equivalences.

We can also consider the analogous functors between the categories of $\mathcal{A}_p$-modules and of $eH_pe$-modules.
The functor $\Gamma: \mathcal{A}_p\operatorname{-mod}\rightarrow eH_pe\operatorname{-mod}$ is the composition
of the Morita equivalence between $\mathcal{H}_p$ and $\mathcal{A}_p$, of $\Gamma:\mathcal{H}_p\operatorname{-mod}\rightarrow H_p\operatorname{-mod}$, and of the functor $M\mapsto eM: H_p\operatorname{-mod} \rightarrow eH_pe\operatorname{-mod}$. A parameter $p$ is called {\it spherical} if the latter is an equivalence and {\it aspherical} else. The aspherical locus was determined by Dunkl and Griffeth in \cite{DG}. Namely, they have shown that the aspherical locus is the union of the hyperplanes
\begin{itemize}
\item
$\kappa=\frac{r}{s}$,  where $0<r\leqslant s, 1<s\leqslant n$,
\item
$ N/\ell=h_{i-N} - h_{i} + m\kappa$, where $i$ is an index, $m$ is an integer with $|m|<n$, $N$ is an integer not divisible by $\ell$ such that $$1\leqslant N\leqslant i + \left( \sqrt{n+\frac{1}{4}m^2} - \frac{1}{2}m-1\right) \ell,$$
and $h_{i-N}$ stands for $h_j$ with $0\leqslant j\leqslant \ell-1$ and $i-N-j$ divisible by $\ell$.
\end{itemize}

\subsection{Main result and corollaries}\label{SS:main_res}
The following is the main result of this paper.
\begin{Thm}\label{Thm:equiv}
Assume $\kappa\neq 0$.
Suppose the orders $\leqslant^\theta$ and $\leqslant^p$ are refined by a common partial order $\leqslant$
(meaning that $\lambda\leqslant^\theta \mu\Rightarrow \lambda\leqslant \mu, \lambda\leqslant^p \mu\Rightarrow
\lambda\leqslant \mu$). Then the functors $\Gamma: \mathcal{H}_p\operatorname{-mod}\rightleftarrows
H_p\operatorname{-mod}: \operatorname{Loc}$ are mutually inverse category equivalences. Consequently, $\Gamma: \mathcal{A}_p\operatorname{-mod}
\rightarrow eH_pe\operatorname{-mod}$ is a quotient functor.
\end{Thm}

The following corollary answers a question of Jenkins, \cite{Jenkins}, in affirmative.

\begin{Cor}\label{Cor:equiv1} We still assume that $\kappa\neq 0$. Then the following holds.
\begin{enumerate}
\item
For every $p$, there is $\theta$ such that $\Gamma: \mathcal{H}_p\operatorname{-mod}\rightleftarrows H_p\operatorname{-mod}:\Loc$
are mutually inverse category equivalences.
\item Suppose that, for the parameters $p,p'$ with integral difference (meaning that $p-p'$ can be represented by an $\ell+1$-tuple
of integers) and $\kappa,\kappa'\neq 0$,
the orders $\leqslant^{p},\leqslant^{p'}$ are equivalent. Then
there is an equivalence $O_p\rightarrow O_{p'}$ that maps $\Delta_{p}(\lambda)$ to $\Delta_{p'}(\lambda)$
and intertwines the KZ functors of \cite{GGOR}.
\end{enumerate}
\end{Cor}



In fact, using an approach of Rouquier, \cite{rouqqsch}, extended in \cite{VV_proof} one can give another proof
of this corollary (under somewhat weaker assumptions).

\begin{Rem}
In \cite[Conjecture 9.2]{BL} Bezrukavnikov and myself produced a conjectural description of the locus, where
the abelian localization holds for a quantization of a quiver variety of a finite or affine type.
That conjecture should follow from Theorem \ref{Thm:equiv} (one still needs to treat the $\kappa=0$
case separately) and some combinatorial argument. We do not pursue that in the present paper.
\end{Rem}

%

\section{Proof of the main theorem}\label{S:proof}
\subsection{Standard objects in $\OCat_p$}\label{S:stand}
In this section, motivated by \cite{VV_proof}, we are going to define objects $\Delta_p^{loc}(\lambda)$ in $\OCat_p$ and show that they are standard objects in
that highest weight category. Moreover, we will see that $R\Gamma(\Delta_p^{loc}(\lambda))=\Delta_p(\lambda)$ and
$L\Loc(\Delta_p(\lambda))=\Delta_p^{loc}(\lambda)$. As a corollary, $\Gamma, \Loc$ are exact on standardly filtered objects
in $O_p,\OCat_p$.

Recall that $\Delta_p(\lambda)$ is defined as follows. We set $\Delta_p:= H_p/H_p\h$. Then $\Delta_p=\Delta e_\lambda$, where
$e_\lambda\in \C W$ stands for the idempotent corresponding to the irreducible $W$-module $\lambda$. We set $\Delta_p^{loc}:=\mathcal{H}_p/\mathcal{H}_p \h$ and $\Delta_p^{loc}(\lambda):=\Delta_p^{loc} e_\lambda$.

Similarly to \cite[Section 9.2]{VV_proof}, we get the following lemma.

\begin{Lem}\label{Lem:Delta_prop}
The following is true:
\begin{enumerate}
\item $\mathcal{H}_p$ is flat over $S(\h)$.
\item $H^0(\M^\theta, \Delta_p^{loc}(\lambda))=\Delta_p(\lambda), H^i(\M^\theta, \Delta_p^{loc}(\lambda))=0$ for $i>0$.
\item $L\Loc(\Delta_p(\lambda))=\Delta^{loc}_p(\lambda)$.
\end{enumerate}
\end{Lem}

The next proposition is the place where we need to use a compatibility between $p$ and $\theta$.

\begin{Prop}\label{Prop:stand}
Suppose that the orders $\leqslant^\theta, \leqslant^p$ are refined by some common partial order on $\operatorname{P}_\ell(n)$.
The object $\Delta^{loc}_p(\lambda)$ is the standard in $\OCat_p$ marked by $\lambda$.
\end{Prop}
\begin{proof}
Recall that  $Y_\lambda:=\{y\in \M^\theta| \lim_{t\rightarrow 0}t.y=z_\lambda\}$,  where $z_\lambda$
means a $T_2$-stable point marked by $\lambda$. By the results of Bezrukavnikov and
Finkelberg, \cite{BF}, and our choice of $\mathcal{P}$, the sheaf $[\mathcal{P}^*/\mathcal{P}^*\h]e_\lambda$ is supported
on $\cup_{\mu\leqslant^\theta \lambda} Y_\mu$. It follows that $[\mathcal{E}nd(\mathcal{P})/
\mathcal{E}nd(\mathcal{P})\h]e_\lambda$ and hence $\Delta^{loc}_p(\lambda)$ is supported on
the same subvariety. By the construction  of the standard objects in $\mathcal{O}_p(\lambda)$
(see \cite{BLPW}), what remains to check that is that $\Delta_p(\lambda)$ is the
indecomposable projective object marked by $\lambda$ in the full subcategory of $\OCat_p(\leqslant \lambda)\subset \OCat_p$ consisting of all objects supported on $\cup_{\mu\leqslant^\theta \lambda}Y_\mu$.

The proof is by induction on $\lambda$. Suppose we already know that $\Delta_p^{loc}(\mu)$ is a standard object in $\OCat_p$
marked by $\mu$, whenever $\mu<^\theta \lambda$. Then what needs to be checked is that $\Ext^i(\Delta_p^{loc}(\lambda),\Delta_p^{loc}(\mu))=0$, and $\dim \Hom(\Delta_p^{loc}(\lambda), \Delta_p^{loc}(\mu))=\delta_{\lambda\mu}$  for all $\mu\leqslant^\theta \lambda$.
But since we know that $R\Gamma$ is a derived equivalence and $R\Gamma(\Delta_p^{loc}(\lambda))=\Delta_p(\lambda)$,
it is enough to show the equalities above for $\Delta_p(\lambda), \Delta_p(\mu)$ instead of $\Delta_p^{loc}(\lambda), \Delta_p^{loc}(\mu)$.
This follows from the condition that $\leqslant^\theta$ is refined by some highest weight order for $O_p$.
\end{proof}

\begin{Cor}
The functors $\Gamma,\Loc$ preserve the subcategories of standardly filtered objects in $O_p,\OCat_p$
and are exact on these subcategories.
\end{Cor}

\subsection{Localization theorem for categories O}\label{S:cat_O}
\begin{Prop}
Under the assumptions in Proposition \ref{Prop:stand},
the functors $\Gamma: \OCat_p\leftrightarrows O_p:\Loc$ are mutually quasi-inverse
equivalences of highest weight categories.
\end{Prop}
\begin{proof}
As we have seen above, we have $\Gamma(\Delta^{loc}_p(\lambda))=\Delta_p(\lambda)$ and $\Loc(\Delta_p(\lambda))=\Delta_p^{loc}(\lambda)$.
Also we have seen that $\Gamma,\Loc$ are exact on the subcategories of standardly filtered objects. Finally, we know that $\Gamma$
is right adjoint to $\Loc$.

We claim that this implies that $\Gamma,\Loc$ are mutually inverse equivalences. To prove the claim, it is enough
to check that $\Loc(P_p(\lambda))$ is projective, where we write $P_p(\lambda)$ for the indecomposable projective in $O_p$
 marked by  $\lambda$ (the claims in the first paragraph imply that $\Gamma\circ \Loc$ is an isomorphism on standardly filtered objects).
Indeed, suppose that  $\Loc(P_p(\lambda))$ is projective for any $\lambda$. From the previous paragraph, it follows that 
the canonical morphism of functors $\operatorname{id}\rightarrow \Gamma\circ \Loc$ is the identity on standardly filtered objects. 
The classes of $\Loc(P_p(\lambda))$ span $K_0(\OCat_p)$. It follows that $\bigoplus_\lambda \Loc(P_p(\lambda))$ is a projective
generator of $\OCat_p$. So $\Gamma,\Loc$ induce  mutually quasiinverse equivalences between $O_p\operatorname{-proj}$
and $\mathcal{O}_p\operatorname{-proj}$. 
 
Since $\OCat_p$ is a highest weight category, the claim that $\Loc(P_p(\lambda))$ is projective is equivalent to $\operatorname{Ext}^1(\Loc(P_p(\lambda)), \Delta_p^{loc}(\mu))=0$ for all $\mu$. Include $\Delta_p^{loc}(\mu)$ into the indecomposable tilting $T_p^{loc}(\mu)$ marked by $\mu$ and let
$C$ stands for the cokernel. The object $\Loc(P_p(\lambda))$ is standardly filtered so $\operatorname{Ext}^1(\Loc(P_p(\lambda)), T^{loc}_p(\mu))=0$. So we just need to prove that $\Hom(\Loc(P_p(\lambda)), T_p(\lambda))\twoheadrightarrow \Hom(\Loc(P_p(\lambda)), C)$.
Since $\Gamma$ is exact on standardly filtered objects, we get $\Gamma(T_p^{loc}(\mu))\twoheadrightarrow
\Gamma(C)$. Since $P_p(\lambda)$ is projective and $\Gamma$ is right adjoint to $\Loc$, we get
$\Hom(\Loc(P_p(\lambda)), T^{loc}_p(\mu))=\Hom(P_p(\lambda), \Gamma(T_p^{loc}(\mu)))
\twoheadrightarrow \Hom(P_p(\lambda), \Gamma(C))=\Hom(\Loc(P_p(\lambda)),C)$, and we are done.
\end{proof}

\subsection{Proof of Theorem \ref{Thm:equiv}}\label{S:thm_proof}
At this point we know that the functors $\Gamma,\Loc$ are mutually quasiinverse equivalences on the categories $\mathcal{O}$.
In this section we deduce that they are mutually inverse equivalences between $\mathcal{H}_p\operatorname{-mod}$ and
$H_p\operatorname{-mod}$.  Consider the $\mathcal{A}_{p+\chi}$-$\mathcal{A}_p$-bimodule $\mathcal{A}_{p,\chi}$
that quantizes the line bundle $\mathcal{O}_\chi$, see, for example, \cite[Section 5]{GL} or \cite[Section 5.1]{BPW}
for a definition. Let $\mathcal{H}_{p,\chi}$ denote the $\mathcal{H}_{p+\chi}$-$\mathcal{H}_p$-bimodule
obtained from $\mathcal{A}_{p,\chi}$ by using the Morita equivalence. Let $H_{p,\chi}$ denote the global sections of
this bimodule, this is a $H_{p+\chi}$-$H_p$-bimodule.

Similarly to \cite[Section 5.3]{BPW}, the claim that $\Gamma:\mathcal{H}_p\operatorname{-mod}\rightarrow H_p\operatorname{-mod}$ is an equivalence, is equivalent to saying that the translation bimodules $H_{p,\chi}, H_{p+\chi,-\chi}$ define mutually inverse Morita equivalences for $\chi$
being a sufficiently dominant element of $\operatorname{Pic}(\M^\theta)$. Similarly, the localization theorem holds for categories
$\mathcal{O}$ if the tensor products with $H_{p,\chi}, H_{p+\chi,-\chi}$ define mutually inverse equivalences between
$O_p, O_{p+\chi}$. Set $B:=H_{p+\chi,-\chi}\otimes_{H_{p+\chi}}H_{p,\chi}$. This bimodule
has a natural homomorphism to $H_p$. We know that the induced homomorphism $B\otimes_{H_p} N\rightarrow N$ is an isomorphism for
all $N\in O_p$. We claim that this implies that $M\cong H_p$. A key step in the proof is the following result
that follows from  \cite[Theorem 2.2.2]{Ginzburg_irr}.

\begin{Lem}\label{Gizb_lem}
For any primitive ideal $J\subset H_p$, there is an irreducible module $N\in O_p$ such that $J$ coincides with
the annihilator of $N$.
\end{Lem}

First, we claim that $B\twoheadrightarrow H_p$. Indeed, otherwise the image is contained in some primitive ideal $J$.
Let $N\in O_p$ be a simple annihilated by that ideal. Then $B\otimes_{H_p}N\rightarrow N$ factors through
$J\otimes_{H_p}N\rightarrow N$ which is zero. So, indeed, $B\twoheadrightarrow H_p$.

Now let $K$ be the kernel of $B\twoheadrightarrow H_p$. The sequence $0\rightarrow K\otimes_{H_p}N\rightarrow B\otimes_{H_p}N
\rightarrow N\rightarrow 0$ is exact. It follows that $K\otimes_{H_p}N=0$ for all $N$ in the category $O_p$.
We claim that this implies $K=0$. Assume the converse. For the purposes of the proof of the claim, 
we may assume that $K$ is irreducible. Let $J$ be its right annihilator,
let $\mathcal{L}$ be the leaf in $\h\oplus \h^*/W$ that is dense in the associated variety of $H_p/J$, and let $\underline{W}$ be the corresponding parabolic subgroup. Let $N$ be a simple in $O_p$ whose annihilator is precisely $J$. It follows 
that the support of $L$ in $\h/W\subset \h\oplus \h^*/W$ lies in $\overline{\mathcal{L}}$ and intersects $\mathcal{L}$. 
Apply the  restriction functors $\bullet_{\dagger}$ corresponding to $\underline{W}$ from \cite{BE,sraco} to $N,K$, respectively.
Then we get (in the notation of \cite{sraco}) that 
$0=(K\otimes_{H_p}N)_{\dagger}=K_{\dagger}\otimes_{\underline{H}_p^+}N_{\dagger}$.
But $N_{\dagger}$ is a nonzero  $\underline{H}_p^+$-module annihilated
by $J_{\dagger}$ (and hence finite dimensional and semisimple since $J_{\dagger}$ is a $N_{W}(\underline{W})$-stable maximal ideal of finite codimension
by \cite{sraco}). Also $K_{\dagger}$ is finite dimensional and its right annihilator is precisely $J_{\dagger}$.
It follows that $K_{\dagger}\otimes_{\underline{H}_p^+}N_{\dagger}$ is nonzero, a contradiction.

So we get that $B=H_p$. Similarly, $H_{p,\chi}\otimes_{H_p}H_{p+\chi,-\chi}=H_{p+\chi}$, and we are done.

\subsection{Proof of Corollary \ref{Cor:equiv1}}\label{SS:proof_Cor1}
{\it Step 1: deformation of parameters}. We claim that there is a Zariski dense set of parameters $p'=(\kappa',h_0',\ldots, h_{\ell-1}')$
such that the orders $\leqslant^{p}$ and $\leqslant^{p'}$ coincide and $\kappa'-\kappa, h_i'-h_i\in \Z$.

Assume first that $\kappa\in \mathbb{Q}$.   Let us write $h_i$
as $\kappa s_i-i/\ell$. The condition that boxes $b=(x,y,i)$ and $b'=(x',y',i')$ are equivalent means
that $(s_i+x-y)-(s_{i'}+x'-y')\in \frac{1}{\kappa}\Z$, while the inequality $b<b'$ means that either
$(s_i+x-y)-(s_{i'}+x'-y')\in \frac{1}{\kappa}\Z_{>0}$ or  $s_i+x-y=s_{i'}+x'-y'$ and $i<i'$. Pick a
sufficiently large positive integer $M$ such that $(M-1)\kappa\in \Z$. Replace $\kappa$ with $\kappa'=M\kappa$. The condition $(s_i+y-x)-(s_{i'}+y'-x')\in \frac{1}{M\kappa}\Z$ is still true. Pick integers $m_0<m_1<\ldots<m_{\ell-1}\ll M$ and set $s_i':=s_i-\frac{m_i}{M\kappa}$.
Then set $p'=(\kappa, \kappa's_0', \kappa's_1'-1/\ell,\ldots, \kappa's_{\ell-1}'-(\ell-1)/\ell)$. It is clear
from the previous description that passing from $p$ to $p'$ does not change the ordering on boxes and hence
the orders $\leqslant^{p}$ and $\leqslant^{p'}$ are equivalent. Also the set of possible $p'$ is Zariski dense.

Now suppose $\kappa\not\in \mathbb{Q}$. Let $\kappa':=\kappa+n$ for  $n\in \Z$. Then let us split the indexes
$\{0,\ldots,\ell-1\}$ into equivalence classes: $i\sim j$ if $s_i-s_j\in \kappa^{-1}\Z$. We introduce new
parameters $s_i^0$ such that $\kappa'(s_i^0-s^0_j)=\kappa(s_i-s_j)$ for any pair of equivalent $i,j$.
Then take integers $m_0<m_1\ldots <m_{\ell-1}$ and set $s_i':=s_i^0-m_i/\kappa'$. Then again the order on boxes
is preserved and hence $\leqslant ^p$ and $\leqslant^{p'}$ are equivalent.


{\it Step 2: relation with the Morse function}.  If the parameter $p=(\kappa, h_0,\ldots,h_{\ell-1})$ has real entries, then 
we set $\theta_p=(-\kappa+ h_0-h_{\ell-1},h_1-h_0,\ldots,h_{\ell-1}-h_{\ell-2})$.
As was checked in \cite{Gordon}, when $\theta_p$ is generic, the value of $c_p(\lambda)$ coincides
with the image of $z_\lambda$ under the moment map for the action of $\mathbb{S}^1\subset T_1$ on $\M^{\theta_p}$.
From Morse theory, it follows that the order $\leqslant^p_c$ refines the order $\leqslant^{\theta_p}$.

Even if $p$ does not have real entries, it is easy to see that there is a parameter $\tilde{p}$ with real
entries such that $\leqslant^p, \leqslant^{\tilde{p}}$ are equivalent.
Let $\mathcal{F}$ be the set of all affine rational functions that vanish on $\kappa,h_0,\ldots,h_{\ell-1}$
(where we assume that $h_0+\ldots+h_{\ell-1}=0$, recall that these elements are defined up to a common summand).
If we have parameters $p,p'$ such that $\mathcal{F}=\mathcal{F}'$, then the orders coincide. Clearly, we can take
$p'$ with real entries such that $\mathcal{F}=\mathcal{F}'$.

{\it Step 3: proof of part (1)}. Take $\theta=\theta_{p'}$ for some $p'$ as in Step 1. By Step 2,
the assumptions of the theorem are satisfied, and we are done.

{\it Step 4: proof of part (2)}. Now suppose $p,p'$ are as in (2) of the corollary. As we have seen in Step 1,
there is $\theta$ such that the global section functor gives rise to highest weight equivalences
$\OCat_{p}\xrightarrow{\sim} O_p, \OCat_{p'}\xrightarrow{\sim} O_{p'}$. The categories
$\OCat_p$ and $\OCat_{p'}$ are naturally equivalent -- via a tensor product with a quantized
line bundle. By the construction, a resulting equivalence $O_p\xrightarrow{\sim}O_{p'}$ is
highest weight and preserves labels. The claim that it intertwines the KZ functors was essentially proved
in \cite[Section 9.4]{VV_proof}. This finishes the proof.

\section{Conjectural generalization}\label{S:open}
We use the notation from the introduction.
We assume that $X_0$ is a singular Poisson variety, $X$ is a smooth symplectic variety equipped with a Poisson resolution
of singularities $\rho:X\rightarrow X_0$. Moreover, we assume that the two-dimensional torus $T_1\times T_2$ acts on both $X,X_0$
such that
\begin{itemize}
\item $\rho$ is $T_1\times T_2$-equivariant.
\item $T_1$ contracts $X_0$ to a point, called $0$. Further, $T_1$ rescales the symplectic form $\omega$: $t.\omega=t^d \omega$ for some
integer $d>0$.
\item The action of $T_2$ on $X$ is Hamiltonian with finitely many fixed points.
\end{itemize}

As Kaledin proved in \cite{Kaledin}, there is a vector bundle $\mathcal{P}$ on $X$ with the following properties:
\begin{enumerate}
\item The algebra $\tilde{A}:=\operatorname{End}_{\mathcal{O}_X}(\mathcal{P})^{opp}$ has finite homological dimension.
\item $\operatorname{Ext}^i(\mathcal{P},\mathcal{P})=0$ for $i>0$.
\end{enumerate}
As a consequence $\operatorname{RHom}_{\mathcal{O}_X}(\mathcal{P},\bullet)$ defines a derived equivalence between
$D^b(\operatorname{Coh}(X))$ and $D^b(\tilde{A}\operatorname{-mod})$.
It can be shown that this bundle is $T_1\times T_2$-equivariant. Also Kaledin checked that, up to a Morita equivalence,
the algebra $\tilde{A}$ is independent of the choices of $X$ and of $\mathcal{P}$ (and so depends only on $X_0$).

Now we can quantize $\mathcal{P}$ to a right $\mathcal{A}_\lambda$-module $\mathcal{P}_\lambda$. Let $H_\lambda$ be the endomorphism algebra
of that right module. Again, the derived categories $D^b(\mathcal{A}_\lambda\operatorname{-mod})$ and $D^b(H_\lambda\operatorname{-mod})$ are equivalent. If $\mathcal{O}_X$ is a direct summand of
$\mathcal{P}$, then we have the quotient functor of the form $N\mapsto eN$ from $H_\lambda\operatorname{-mod}$
to $A_\lambda\operatorname{-mod}$. But the existence of such direct summand is not known in general.

It is not difficult to show that the action of $T_2$ on $H_\lambda$ is Hamiltonian. This allows to define
the category $O_\lambda$ for $H_\lambda$ and also Verma modules, compare, for example, with \cite{GL,BLPW}.

\begin{Conj}\label{Conj:1}
This category $O_\lambda$ is always highest weight and the Verma modules are precisely the standard objects.
\end{Conj}

As in the Cherednik case, we can establish a natural bijection between the labels in the categories $\mathcal{O}_\lambda$
and $O_\lambda$.

\begin{Conj}\label{Conj:2}
Suppose that one can take a common highest weight order for $O_\lambda$ and $\OCat_\lambda$.
Then the functor $\Hom_{\mathcal{A}_\lambda}(\mathcal{P}_\lambda,\bullet)$ is a highest weight equivalence
between $\OCat_\lambda$ and $O_\lambda$ and also an equivalence between $\mathcal{A}_\lambda\operatorname{-mod}$
and $H_\lambda\operatorname{-mod}$.
\end{Conj}

\end{document}